\documentclass[12pt]{amsart}
\usepackage{tikz, hyperref, amssymb}

\newtheorem{thm}{Theorem}[section]
\newtheorem{cor}[thm]{Corollary}
\newtheorem{lem}[thm]{Lemma}
\newtheorem{prop}[thm]{Proposition}

\theoremstyle{definition}
\newtheorem{dfn}[thm]{Definition}

\theoremstyle{remark}
\newtheorem{rmk}[thm]{Remark}

\newtheorem{example}[thm]{Example}

\newcommand{\CC}{\mathbb{C}}

\newcommand{\NN}{\mathbb{N}}

\newcommand{\TT}{\mathbb{T}}
\newcommand{\ZZ}{\mathbb{Z}}
\newcommand{\RR}{\mathbb{R}}

\newcommand{\Gg}{\mathcal{G}}

\newcommand{\Kk}{\mathcal{K}}
\newcommand{\Ll}{\mathcal{L}}
\newcommand{\Mm}{\mathcal{M}}
\newcommand{\Pp}{\mathcal{P}}

\newcommand{\Tt}{\mathcal{T}}
\newcommand{\Uu}{\mathcal{U}}
\newcommand{\Vv}{\mathcal{V}}
\newcommand{\Ww}{\mathcal{W}}

\newcommand{\Zz}{\mathcal{Z}}

\newcommand{\id}{\operatorname{id}}

\newcommand{\Aut}{\operatorname{Aut}}

\newcommand{\MCE}{\operatorname{MCE}}

\newcommand{\lsp}{\operatorname{span}}
\newcommand{\clsp}{\overline{\lsp}}

\newcommand{\Ad}{\operatorname{Ad}}

\newcommand{\lt}{\operatorname{lt}}

\newcommand{\Hcat}[1]{\underline{H}^{#1}}
\newcommand{\Zcat}[1]{\underline{Z}^{#1}}
\newcommand{\dcat}[1]{\underline{\delta}^{#1}}

\newcommand{\csLA}[1]{C^*(\Lambda, A, #1)}

\title[$K$-theory of twisted $k$-graph algebras]{On the $K$-theory of twisted higher-rank-graph $C^*$-algebras}

\author{Alex Kumjian}
\address{Alex Kumjian\\ Department of Mathematics (084)\\ University
of Nevada\\ Reno NV 89557-0084\\ USA} \email{alex@unr.edu}

\author{David Pask}
\address{David Pask, Aidan Sims\\ School of Mathematics and
Applied Statistics  \\
University of Wollongong\\
NSW  2522\\
AUSTRALIA} \email{dpask, asims@uow.edu.au}
\author{Aidan Sims}

\thanks{This research was supported by the Australian Research Council.}

\date{\today}

\begin{document}

\begin{abstract}
We investigate the $K$-theory of twisted higher-rank-graph algebras by adapting parts of
Elliott's computation of the $K$-theory of the rotation algebras. We show that each
2-cocycle on a higher-rank graph taking values in an abelian group determines a
continuous bundle of twisted higher-rank graph algebras over the dual group. We use this
to show that for a circle-valued $2$-cocycle on a higher-rank graph obtained by
exponentiating a real-valued cocycle, the $K$-theory of the twisted higher-rank graph
algebra coincides with that of the untwisted one.
\end{abstract}

\maketitle

\section{Introduction and preliminaries}

Higher-rank graphs, or $k$-graphs, and their $C^*$-algebras were introduced by the first
two authors in \cite{KP2000} as a common generalisation of the graph algebras of
\cite{KPR98} and the higher-rank Cuntz-Krieger algebras of \cite{RobSte:JraM99}. Since
then these $C^*$-algebras have been studied by a number of authors (see for example
\cite{DavidsonYang:CJM09, Evans:NYJM08, FarthingMuhlyEtAl:SF05, SkalskiZacharias:HJM08}).

In \cite{KPS3, KPS4} we introduced homology and cohomology theories for $k$-graphs, and
showed how a $\TT$-valued 2-cocycle $c$ on a $k$-graph $\Lambda$ can be incorporated into
the relations defining its $C^*$-algebra to obtain a twisted $k$-graph $C^*$-algebra
$C^*(\Lambda, c)$. We showed that elementary examples of this construction include all
noncommutative tori and the twisted Heegaard-type quantum spheres of \cite{BHMS}.

In this paper we consider the $K$-theory of twisted $k$-graph $C^*$-algebras, following
the approach of Elliott to the computation of $K$-theory for the noncommutative tori
\cite{Elliott:Neptun84}. Rather than computing the $K$-theory of a twisted $k$-graph
$C^*$-algebra directly, we show that in many instances the problem can be reduced to that
of computing the $K$-theory of the untwisted $C^*$-algebra of the same $k$-graph. In
many examples, the latter is known or readily computable (see \cite{Evans:NYJM08}).
In order to follow Elliott's method, we must first construct a suitable $C^*$-algebra bundle from the
data used to build a twisted $k$-graph $C^*$-algebra.

In Section~\ref{sec:extCstar} we consider an abelian group $A$, and an $A$-valued
$2$-cocycle on a $k$-graph $\Lambda$. To this data we associate a $C^*$-algebra
$\csLA{c}$. We show in Proposition~\ref{prp:extension algebra bundle} that $\csLA{c}$ is
a $C_0(\widehat{A}\,)$-algebra whose fibre over a character $\chi$ of $A$ is the twisted
$k$-graph $C^*$-algebra $C^*(\Lambda, \chi \circ c)$. In particular, $\csLA{c}$ is the
section algebra of an upper semicontinuous $C^*$-bundle over $\widehat{A}$. In
Section~\ref{app:lsc}, we adapt an argument of Rieffel to show that this bundle is
continuous (Corollary~\ref{cor:cts field}). We demonstrate that examples of this
construction include continuous fields of noncommutative tori and of Heegaard-type
quantum spheres.

In Section~\ref{sec:AFbundle} we consider $k$-graphs $\Lambda$ for which the degree map
is a coboundary. We show that for every $A$-valued $2$-cocycle on such a $k$-graph, we
have $\csLA{c} \cong C^*(A) \otimes C^*(\Lambda)$. If $A = \RR$ then $\widehat{A} \cong
\RR$ and so $C^*(\Lambda , \RR , c)$ is a $C_0(\RR)$-algebra. We deduce in
Corollary~\ref{cor:[0,1]-bundle} that the ``restriction" $C^*(\Lambda, \RR, c)|_{[0,1]}$
is isomorphic to $C([0,1]) \otimes C^*(\Lambda)$, and so the maps $C^*(\Lambda, \RR,
c)|_{[0,1]} \to C^*(\Lambda, \RR, c)_t \cong C^*(\Lambda)$ all induce isomorphisms in
$K$-theory.

In Section~\ref{sec:K-theory} we execute Elliott's argument. We consider an $\RR$-valued
$2$-cocycle $c_0$ on a $k$-graph $\Lambda$. The degree functor on $\Lambda \times_d
\ZZ^k$ is a coboundary, and we exploit this to show that $C^*(\Lambda, \RR, c_0)_{[0,1]}$
embeds as a full corner of a crossed product of $C([0,1]) \otimes C^*(\Lambda \times_d
\ZZ^k)$ by a fibre-preserving $\ZZ^k$-action. An induction on $k$ shows that the
homomorphisms of this crossed product onto its fibres induce isomorphisms in $K$-theory.
Hence the $K$-theory of $C^*(\Lambda, \RR, c_0)_{1} \cong C^*(\Lambda, e^{\imath tc_0})$
is identical to that of $C^*(\Lambda, \RR, c_0)_{0} \cong C^*(\Lambda)$
(Theorem~\ref{thm:main K-theory}). We apply our results to the noncommutative tori and to
the twisted Heegaard-type quantum spheres of \cite{BHMS}, recovering the existing
formulae for their $K$-theory. We also discuss the implications of our results for
Kirchberg algebras associated to $k$-graphs.

\subsection*{Background and notation}
Throughout the paper, we regard $\NN^k$ as a monoid under addition, with identity 0 and
generators $e_1, \dots, e_k$. For $m,n \in \NN^k$, we write $m_i$ for the
$i$\textsuperscript{th} coordinate of $m$, and define $m \vee n \in \NN^k$ by $(m \vee
n)_i = \max\{m_i, n_i\}$.

Let $\Lambda$ be a countable small category and $d : \Lambda \to \NN^k$ be a functor.
Write $\Lambda^n := d^{-1}(n)$ for $n \in \NN^k$. Recall that $\Lambda$ is a $k$-graph if
$d$ satisfies the \emph{factorisation property}: $(\mu,\nu) \mapsto \mu\nu$ is a
bijection of $\{(\mu,\nu) \in \Lambda^m \times \Lambda^n : s(\mu) = r(\nu)\}$ onto
$\Lambda^{m+n}$ for each $m,n \in \NN^k$. We then have $\Lambda^0 = \{\id_o : o\in
\operatorname{Obj}(\Lambda)\}$, and so we regard the domain and codomain maps as maps $s,
r : \Lambda \to \Lambda^0$. Recall from \cite{PaskQuiggEtAl:NYJM04} that for $v,w \in
\Lambda^0$ and $X \subseteq \Lambda$, we write
\[
v X := \{\lambda \in X : r(\lambda) = v\},\quad
X w := \{\lambda \in X : s(\lambda) = w\},\quad\text{and}\quad
v X w = vX \cap Xw.
\]
A $k$-graph is \emph{row-finite with no sources} if $0 < |v\Lambda^n| < \infty$ for all
$v \in \Lambda^0$ and $n \in \NN^k$. Given $\mu,\nu \in \Lambda$, we write
$\MCE(\mu,\nu)$ for the set $\{\lambda \in \Lambda^{d(\mu) \vee d(\nu)} : \lambda =
\mu\mu' = \nu\nu'\text{ for some }\mu',\nu'\}$. See \cite{KP2000} for further details
regarding the basic structure of $k$-graphs. Let $\Lambda^{*2} = \{(\lambda, \mu) \in
\Lambda \times \Lambda : s(\lambda) = r(\mu)\}$ denote the collection of composable pairs
in $\Lambda$.

Given an abelian group $A$ and a $k$-graph $\Lambda$, an $A$-valued
$2$-cocycle\footnote{In \cite{KPS4} these were called \emph{categorical cocycles}, in
contradistinction to cubical cocycles.} $c$ on $\Lambda$ is a map $c :\Lambda^{*2} \to A$
such that $c(\lambda, s(\lambda)) = c(r(\lambda), \lambda) = 0$ for all $\lambda$ and
$c(\mu,\nu) + c(\lambda, \mu\nu) = c(\lambda,\mu) + c(\lambda\mu, \nu)$ for each
composable triple $(\lambda,\mu,\nu)$. The group of all such $2$-cocycles is denoted
$\Zcat2(\Lambda, A)$. Given $c \in \Zcat2(\Lambda, \TT)$, the twisted $k$-graph
$C^*$-algebra $C^*(\Lambda, c)$ is the universal $C^*$-algebra generated by elements
$s_\lambda$, $\lambda \in \Lambda$ such that: (1) the $s_v$, $v \in \Lambda^0$ are
mutually orthogonal projections; (2) $s_\mu s_\nu = c(\mu,\nu) s_{\mu\nu}$ when $\mu,
\nu$ are composable; (3) $s^*_\mu s_\mu = s_{s(\mu)}$ for all $\mu$; and (4) $s_v =
\sum_{\lambda \in v\Lambda^n} s_\lambda s^*_\lambda$ for all $v \in \Lambda^0$ and $n \in
\NN^k$. See \cite{KPS4} for further details regarding twisted $k$-graph $C^*$-algebras.
We denote by $\Zz(B)$, $\Mm(B)$ and $\Uu(B)$ the centre, multiplier algebra and unitary
group of a $C^*$-algebra $B$.

\section{\texorpdfstring{$C^*$}{C*}-bundles associated to 2-cocycles}\label{sec:extCstar}

Given a $k$-graph $\Lambda$ and an abelian group $A$, we associate to each $c \in
\Zcat2(\Lambda, A)$ a $C^*$-algebra $\csLA{c}$. We show that $\csLA{c}$ is a
$C_0(\widehat{A}\,)$-algebra whose fibres are twisted $k$-graph $C^*$-algebras associated
to $\Lambda$.

Throughout this paper, if $A$ is an abelian group, then given a unitally extendible
homomorphism $\pi : C^*(A) \to B$ we also write $\pi :\Mm(C^*(A)) \to \Mm(B)$ for its
extension. We identify elements of $A$ with the corresponding unitary multipliers of
$C^*(A)$. We denote general elements of $C^*(A)$ by the letters $f,g,\dots$ and often
regard them as elements of $C_0(\widehat{A}\,)$.

\begin{dfn}\label{dfn:c-representation}
Let $\Lambda$ be a row-finite $k$-graph with no sources, let $A$ be an abelian group, and
fix $c \in \Zcat2(\Lambda, A)$. A \emph{$c$-representation} $(\phi, \pi)$ of $(\Lambda,
A)$ in a $C^*$-algebra $B$ consists of a map $\phi : \Lambda \to \Mm(B)$ and a unitally
extendible homomorphism $\pi : C^*(A) \to \Mm(B)$ such that $\phi(\lambda)\pi(f) \in B$
for all $\lambda \in \Lambda$ and $f \in C^*(A)$, and:
\begin{enumerate}\renewcommand{\theenumi}{R\arabic{enumi}}
\item\label{it:LArep1} $\phi(\lambda)\pi(f) = \pi(f)\phi(\lambda)$ for all $\lambda
    \in \Lambda$ and $f \in C^*(A)$;
\item\label{it:LArep2} $\phi(v)$, $v \in \Lambda^0$ are mutually orthogonal
    projections and $\sum \phi(v) \to 1_{\Mm(B)}$ strictly;
\item\label{it:LArep3} $\phi(\lambda)\phi(\mu) = \pi(c(\lambda,\mu))\phi(\lambda\mu)$
    whenever $s(\lambda) = r(\mu)$;
\item\label{it:LArep4} $\phi(\lambda)^*\phi(\lambda) = \phi(s(\lambda))$ for all
    $\lambda \in \Lambda$;
\item\label{it:LArep5} $\sum_{\lambda \in v\Lambda^n} \phi(\lambda) \phi(\lambda)^* =
    \phi(v)$ for all $v \in \Lambda^0$ and $n \in\NN^k$.
\end{enumerate}
We define $C^*(\phi, \pi) := C^*(\{\phi(\lambda)\pi(f) : \lambda \in \Lambda, f \in
C^*(A)\} \subseteq B$.
\end{dfn}

\begin{rmk}
If $A$ is discrete, $c \in \Zcat2(\Lambda, A)$, and $(\phi, \pi)$ is a $c$-representation
of $(\Lambda, A)$, then $\phi(\Lambda) \subseteq C^*(\phi, \pi)$. If $|\Lambda^0| <
\infty$, then $C^*(\Lambda)$ is unital, so $\pi(C^*(A)) \subseteq C^*(\phi, \pi)$ too;
moreover condition~(\ref{it:LArep1}) then reduces to $\pi(a) \in \Uu\Zz(C^*(\phi, \pi))$
for each $a \in A$.
\end{rmk}

\begin{lem}
Let $\Lambda$ be a row-finite $k$-graph with no sources, let $A$ be an abelian group, and
suppose that $c \in \Zcat2(\Lambda, A)$. There is a $c$-representation $(i_\Lambda, i_A)$
of $(\Lambda, A)$ in a $C^*$-algebra $\csLA{c}$ that is universal in the sense that
\begin{equation}\label{eq:univ spanners}
\csLA{c} = \clsp\{i_\Lambda(\lambda)i_A(f)i_\Lambda(\mu)^* : \lambda,\mu \in \Lambda, f \in C_c(\widehat{A}\,)\},
\end{equation}
and any $c$-representation $(\phi, \pi)$ of $(\Lambda, A)$ in a $C^*$-algebra $B$ induces
a homomorphism $\phi \times \pi : \csLA{c} \to B$ such that $(\phi \times \pi) \circ
i_\Lambda = \phi$ and $(\phi \times \pi) \circ i_A = \pi$.
\end{lem}
\begin{proof}
Given a $c$-representation $(\phi, \pi)$ of $(\Lambda, A)$ in a $C^*$-algebra $B$,
calculations like those of \cite[Lemma~7.2]{KPS4} show that for $\mu, \nu \in \Lambda$,
we have
\begin{equation}\label{eq:multiplication}
\phi(\mu)^*\phi(\nu) = \sum_{\mu\alpha = \nu\beta \in \MCE(\mu,\nu)}
    \pi(c(\nu,\beta) - c(\mu,\alpha)) \phi(\alpha) \phi(\beta)^*.
\end{equation}
Since $\pi(C^*(A))$ is central, $\clsp\{\phi(\lambda)\pi(f)\phi(\mu)^* : \lambda,\mu \in
\Lambda, f \in C_c(\widehat{A}\,)\}$ is therefore a $C^*$-subalgebra of $B$. The
$\phi(\lambda)$ are partial isometries by (\ref{it:LArep2})~and~(\ref{it:LArep4}). Thus
$\|\phi(\lambda)\pi(f)\phi(\mu)^*\| \le \|f\|$ for each $\lambda$. An argument like
\cite[Proposition~1.21]{CBMSbook} now proves the lemma.
\end{proof}

Recall that if $\Lambda$ is a $k$-graph, $A$ is an abelian group and $b : \Lambda \to A$,
then the $2$-coboundary $\dcat1b : \Lambda^{*2} \to A$ is given by $\dcat1b(\lambda_1,
\lambda_2) = b(\lambda_1) - b(\lambda_1 \lambda_2) + b(\lambda_2)$.

\begin{lem}\label{eq:cohomologous exts}
Let $\Lambda$ be a row-finite $k$-graph with no sources, and let $A$ be an abelian group.
Suppose that $c, c' \in \Zcat2(\Lambda, A)$ and $b : \Lambda \to A$ satisfy $c - c' =
\dcat1b$. Then there is an isomorphism $\csLA{c} \cong \csLA{c'}$ determined by
$i^c_\Lambda(\lambda)i^c_A(f) \mapsto i^{c'}_A
(b(\lambda))i^{c'}_\Lambda(\lambda)i^{c'}_A(f)$.
\end{lem}

\begin{proof}
Define $\pi := i^{c'}_A$ and $\phi(\lambda) :=
i^{c'}_A(b(\lambda))i^{c'}_\Lambda(\lambda)$. Then $(\phi, \pi)$ is a $c$-representation
of $(\Lambda, A)$ in $\csLA{c'}$. The homomorphism $\phi \times \pi : \csLA{c} \to
\csLA{c'}$ satisfies the desired formula. Symmetry provides an inverse.
\end{proof}

Recall (from \cite[Appendix~C]{Williams:Crossedproducts07} for example) that if $X$ is a
locally compact Hausdorff space, then a $C_0(X)$-algebra is a $C^*$-algebra $B$ equipped
with a homomorphism $i : C_0(X) \to \Zz\Mm(B)$ such that $\clsp\{i(f)b : f \in
C_0(X)\text{ and }b \in B\} = B$. For $x \in X$, we write $I_x$ for the ideal $\{i(f)b :
b \in B, f(x) = 0\}$, we write $B_x := B/I_x$, and $q_x : B \to B_x$ denotes the quotient
map.

\begin{prop}\label{prp:extension algebra bundle}
Let $\Lambda$ be a row-finite $k$-graph with no sources, let $A$ be an abelian group, and
fix $c \in \Zcat2(\Lambda, A)$. Then $\csLA{c}$ is a $C_0(\widehat{A}\,)$-algebra with
respect to $i_A : C^*(A) \to \Mm(\csLA{c})$. For each character $\chi \in \widehat{A}$,
there is a unique homomorphism $\pi_\chi : \csLA{c} \to C^*(\Lambda, \chi \circ c)$ such
that $\pi_\chi(i_\Lambda(\lambda)i_A(f)) = f(\chi) s_\lambda$. This $\pi_\chi$ descends
to an isomorphism $\csLA{c}_\chi \cong C^*(\Lambda, \chi \circ c)$. The map $\chi \mapsto
\|\pi_\chi(b)\|$ is upper semicontinuous for each $b \in \csLA{c}$, and $b \mapsto (\chi
\mapsto \pi_\chi(b))$ is an isomorphism of $\csLA{c}$ onto the section algebra of an
upper-semicontinuous $C^*$-bundle with fibres $C^*(\Lambda, \chi \circ c)$.
\end{prop}
\begin{proof}
Condition~(\ref{it:LArep1}) and equation~\eqref{eq:univ spanners} imply that each
$\csLA{c}$ is a $C_0(\widehat{A}\,)$-algebra with respect to $i_A$. Fix $\chi \in
\widehat{A}$. The map $\chi \circ c$ belongs to $\Zcat2(\Lambda, \TT)$, and $(\lambda
\mapsto s_\lambda, f \mapsto f(\chi))$ is a $c$-representation of $(\Lambda, A)$ in
$C^*(\Lambda, \chi \circ c)$. The universal property of $\csLA{c}$ yields the desired
homomorphism $\pi_\chi$. We have $I_\chi \subseteq \ker(\pi_\chi)$, and so $\pi_\chi$
descends to a homomorphism $\widetilde{\pi}_\chi : \csLA{c}_\chi \to C^*(\Lambda, \chi
\circ c)$. Since $q_\chi(f) = f(\chi)$ for all $\chi \in \widehat{A}$, the partial
isometries $t_\lambda := q_\chi(i_\Lambda(\lambda))$ constitute a Cuntz-Krieger
$(\Lambda, \chi \circ c)$-family in $\csLA{c}_\chi$. Thus the universal property of
$C^*(\Lambda, \chi \circ c)$ gives an inverse for $\widetilde{\pi}_\chi$. Proposition~C.10(a) of
\cite{Williams:Crossedproducts07} implies that $\chi \mapsto \|q_\chi(b)\| =
\|\pi_\chi(b)\|$ is upper semicontinuous. The final statement is proved in the first
paragraph of the proof of \cite[Theorem~C.26]{Williams:Crossedproducts07}.
\end{proof}

We digress to characterise when $\pi \times \phi$ is injective.

\begin{thm}
Let $A$ be an abelian group, and fix $c \in \Zcat2(\Lambda, A)$. Suppose that $(\phi,
\pi)$ is a $c$-representation of $(\Lambda, A)$ in a $C^*$-algebra $B$. Suppose that
there is an action $\beta$ of $\TT^k$ on $B$ such that $\beta_z(\phi(\lambda)\pi(f)) =
z^{d(\lambda)}\phi(\lambda)\pi(f)$ for all $\lambda$, $f$. Suppose that $f \mapsto
\phi(w) \pi(f)$ is injective on $C^*(A)$ for each $w \in \Lambda^0$. Then $\phi \times
\pi : \csLA{c} \to B$ is injective.
\end{thm}
\begin{proof}
We may assume that $B = C^*(\phi, \pi)$, and hence that $\pi$ maps $C^*(A) =
C_0(\widehat{A}\,)$ into $\Zz\Mm(B)$. Thus $B$ is a $C_0(\widehat{A}\,)$-algebra. The series
$\sum_{w \in \Lambda^0} \phi(w)$ is an approximate identity for $B$. Hence each fibre
$B_\chi$ of $B$ is the quotient by the ideal generated by $\{\phi(w)\pi(f) : w \in
\Lambda^0, f \in C_0(\widehat{A}\,), f(\chi) = 0\}$.

Fix $\chi \in \widehat{A}$, and let $q_\chi : B \to B_\chi$ be the quotient map. Since $f
\mapsto \phi(w) \pi(f)$ induces a faithful representation of $C^*(A)$ we have
$q_\chi(\phi(w)) \not= 0$ for all $w \in \Lambda^0$ and $\chi \in \widehat{A}$. The
action $\beta$ descends to an action $\widetilde{\beta}$ of $\TT^k$ on $B_\chi$
satisfying $\widetilde{\beta}_z(q_\chi(\phi(\lambda))) =
z^{d(\lambda)}q_\chi(\phi(\lambda))$. The elements $\{q_\chi(\phi(\lambda)) : \lambda \in
\Lambda\}$ form a Cuntz-Krieger $(\Lambda, \chi \circ c)$-family. Thus the gauge-invariant
uniqueness theorem \cite[Corollary~7.7]{KPS4} implies that $s_\lambda \mapsto
q_\chi(\phi(\lambda))$ determines an isomorphism $C^*(\Lambda, \chi \circ c) \cong B_\chi$.
Proposition~\ref{prp:extension algebra bundle} implies that $\pi_\chi(i_\Lambda(\lambda)
i_A(f)) \mapsto \chi(f)s_\lambda$ determines an isomorphism $\csLA{c}_\chi \cong
C^*(\Lambda, \chi \circ c)$. Composing these two isomorphisms gives an isomorphism
$\csLA{c}_\chi \cong B_\chi$ which carries each $\pi_\chi(i_\Lambda(\lambda) i_A(f))$ to
$q_\chi(\phi(\lambda)\pi(f))$. Thus $\phi \times \pi$ is a fibrewise-isometric
homomorphism of $C_0(\widehat{A}\,)$-algebras; Proposition~C.10(c) of
\cite{Williams:Crossedproducts07} implies that $\phi \times \pi$ is isometric.
\end{proof}

\section{Lower semicontinuity}\label{app:lsc}

We adapt an argument due to Rieffel (see the proof of \cite[Theorem~2.5]{Rieffel:MA89})
to show that the bundles described in Proposition~\ref{prp:extension algebra bundle} are
continuous bundles.

We recall some background regarding $k$-graph groupoids; for more detail, see
\cite{KP2000, KPS4}. Fix a row-finite $k$-graph $\Lambda$ with no sources. There is a
groupoid $\Gg_\Lambda$ with unit space $\Lambda^\infty$, the space of infinite paths in
$\Lambda$, whose elements have the form $(\alpha x, d(\alpha) - d(\beta), \beta x)$ where
$x \in \Lambda^\infty$ and $\alpha,\beta \in \Lambda$ satisfy $s(\alpha) = s(\beta) =
r(x)$. We use lower-case fraktur letters $\mathfrak{a,b,c}$ for elements of
$\Gg_\Lambda$, and the letters $x,y,z$ for elements of $\Gg_\Lambda^{(0)} =
\Lambda^\infty$. For $\mu,\nu \in \Lambda *_s \Lambda := \{(\lambda, \mu) \in \Lambda
\times \Lambda : s(\mu) = s(\nu)\}$, we write $Z(\mu,\nu)$ for the basic compact open
subset $\{(\mu x, d(\mu) - d(\nu), \nu x) : x \in s(\mu) \Lambda^\infty\}$ of
$\Gg_\Lambda$. We showed in \cite[\S6]{KPS4} that there is a countable set $\Pp$ of pairs
$(\mu,\nu) \in \Lambda *_s \Lambda$ such that $(\lambda, s(\lambda)) \in \Pp$ for all
$\lambda \in \Lambda$ and $\Gg_\Lambda = \bigsqcup\{Z(\mu,\nu) : (\mu,\nu) \in \Pp\}$.
Fix $c \in \Zcat2(\Lambda, \TT)$. Lemma~6.3 of \cite{KPS4} yields a continuous groupoid
cocycle $\sigma_c$ on $\Gg_\Lambda$ and an isomorphism $C^*(\Lambda, c) \cong
C^*(\Gg_\Lambda, \sigma_c)$ which carries each $s_\lambda$ to the indicator function
$1_{Z(\lambda, s(\lambda))}$. We write $*_{\sigma_c}$ for the convolution product in
$C^*(\Gg_\Lambda, \sigma_c)$.

Regard $\Zcat2(\Lambda, \TT)$ as a topological subspace of $\TT^{\Lambda^{*2}}$. Let
$\ell^2(s)$ be the Hilbert $C_0(\Gg_\Lambda^{(0)})$-module completion of
$C_c(\Gg_\Lambda)$ with respect to the $C_0(\Gg_\Lambda^{(0)})$-valued inner-product
\[\textstyle
\langle f, g \rangle_{\ell^2(s)}(x) = \sum_{s(\mathfrak{a}) = x} \overline{f(\mathfrak{a})} g(\mathfrak{a}).
\]

Let $R : C^*(\Gg_\Lambda, \sigma_c) \to C_0(\Gg_\Lambda^{(0)})$ be the faithful
conditional expectation such that $R(f) = f|_{\Gg_\Lambda^{(0)}}$ for $f \in
C_c(\Gg_\Lambda)$. The map $\langle f, g \rangle_{\sigma_c} := R(f^* *_{\sigma_c} g)$ is
a $C_0(\Gg_\Lambda^{(0)})$-valued inner product on $C_c(\Gg_\Lambda, \sigma_c)$. The
completion $H(\sigma_c)$ of $C_c(\Gg_\Lambda)$ in the corresponding norm is a
right-Hilbert $C_0(\Gg_\Lambda^{(0)})$-module. Given a right-Hilbert module $H$, let
$\Ll(H)$ denote the algebra of adjointable operators on $H$.

\begin{prop}\label{prp:strong cty}
Let $\Lambda$ be a row-finite $k$-graph with no sources, and suppose $c \in
\Zcat2(\Lambda, \TT)$. Then $\langle f, g\rangle_{\sigma_c} = \langle f,
g\rangle_{\ell^2(s)}$ for all $f, g \in C_c(\Gg_\Lambda)$. There is an injective
homomorphism $\pi_c: C^*(\Gg_\Lambda, \sigma_c)\to\Ll(\ell^2(s))$ such that $\pi_c(f)g =
f *_{\sigma_c} g$ for all $f,g \in C_c(\Gg_\Lambda)$. Suppose that $c_n \to c$ in
$\Zcat2(\Lambda,\TT)$. Then for all compact open bisections $U,V \subseteq \Gg_\Lambda$
and all $\varepsilon > 0$, there exists $N \in \NN$ such that $\|(\pi_{c_n}(1_U) -
\pi_c(1_U)) 1_V\| \le \varepsilon$ for all $n \ge N$. For each compact open bisection $U
\subseteq \Gg_\Lambda$, the map $c \mapsto \pi_c(1_{U})$ is strongly continuous from
$\Zcat2(\Lambda, \TT)$ to $\Ll(\ell^2(s))$.
\end{prop}

\begin{proof}
Fix $c \in \Zcat2(\Lambda, \TT)$ and $x \in \Gg_\Lambda^{(0)}$. Then
\[\textstyle
\langle f, g\rangle_{\sigma_c}(x)
    = \sum_{\mathfrak{a}\mathfrak{b} = x} \sigma_c(\mathfrak{a}, \mathfrak{b})f^*(\mathfrak{a}) g(\mathfrak{b})
    = \sum_{s(\mathfrak{b}) = x} \sigma_c(\mathfrak{b}^{-1}, \mathfrak{b})
        \overline{\sigma_c(\mathfrak{b}^{-1}, \mathfrak{b})}\overline{f(\mathfrak{b})} g(\mathfrak{b}).
\]
Thus $\langle f, g\rangle_{\sigma_c} = \langle f, g\rangle_{\ell^2(s)}$. Hence $\id :
C_c(\Gg_\Lambda) \to C_c(\Gg_\Lambda)$ extends to an isometric isomorphism $\phi_c :
H(\sigma_c) \to \ell^2(s)$.

Multiplication in $C_c(\Gg_\Lambda, \sigma_c)$ extends to a left action of
$C^*(\Gg_\Lambda, \sigma_c)$ on $H(\sigma_c)$. Since the map $\phi_c : H(\sigma_c) \to
\Ll(\ell^2(s))$ is isometric for each $c$, there is a homomorphism $\pi_c :
C^*(\Gg_\Lambda, \sigma_c) \to \Ll(\ell^2(s))$ as claimed. The $\ZZ^k$-valued $1$-cocycle
$\tilde{d} : (x,m,y) \to m$ on $\Gg_\Lambda$ induces a strongly continuous $\TT^k$ action
on $C_c(\Gg_\Lambda)$ (called the gauge action) given by $\gamma_z(f)(\mathfrak{a}) =
z^{\tilde{d}(\mathfrak{a})} f(\mathfrak{a})$ for $f \in C_c(\Gg_\Lambda)$ and $z \in
\TT^k$. Let $\rho$ denote the $\TT^k$ action on $\Ll(\ell^2 (s))$ induced by the gauge
action on $C_c(\Gg_\Lambda)$, then $\rho_z \circ \pi_c = \pi_c \circ \gamma_z$ for all $z
\in \TT^k$. Since $C^*(\Gg_\Lambda, \sigma_c) \cong C^*(\Lambda, c)$, Corollary~7.7
of~\cite{KPS4} implies that $\pi_c$ is faithful.

Fix compact open bisections $U, V$  of $\Gg_\Lambda$. For $\mathfrak{a} \in UV$, let
$\mathfrak{a}^U \in U$ and $\mathfrak{a}^V \in V$ be the unique elements such that
$\mathfrak{a} = \mathfrak{a}^U \mathfrak{a}^V$. Then
\begin{align*}
(\pi_c(1_U) 1_V)(\mathfrak{a})
    &= (1_U *_{\sigma_c} 1_V)(\mathfrak{a}) \\
    &= \sum_{\mathfrak{b}\mathfrak{c} = \mathfrak{a}} \sigma_c(\mathfrak{b},\mathfrak{c}) 1_U(\mathfrak{b})1_V(\mathfrak{c})
    = \begin{cases}
        \sigma_c(\mathfrak{a}^U, \mathfrak{a}^V) &\text{if $\mathfrak{a} \in UV$} \\
        0 &\text{otherwise.}
    \end{cases}
\end{align*}
Recall from \cite[Lemma~6.3]{KPS4} that $(\mu_{\mathfrak{a}^U}, \nu_{\mathfrak{a}^U}),
(\mu_{\mathfrak{a}^V}, \nu_{\mathfrak{a}^V})$ and $(\mu_{\mathfrak{a}},
\nu_{\mathfrak{a}})$ denote the elements of $\Pp$ such that $\mathfrak{a}^U \in
Z (\mu_{\mathfrak{a}^U}, \nu_{\mathfrak{a}^U})$, $\mathfrak{a}^V \in
Z (\mu_{\mathfrak{a}^V}, \nu_{\mathfrak{a}^V})$ and $\mathfrak{a} \in
Z (\mu_{\mathfrak{a}}, \nu_{\mathfrak{a}})$. There exist $\alpha,\beta, \gamma \in
\Lambda$ and $y \in \Lambda^\infty$ such that
\begin{equation}\label{eq:alphabetagamma}
    \begin{split}
    \mathfrak{a}^U = (\mu_{\mathfrak{a}^U}\alpha y, \widetilde{d}({\mathfrak{a}^U}), &\nu_{\mathfrak{a}^U}\alpha y),\quad
    \mathfrak{a}^V = (\mu_{\mathfrak{a}^V}\beta y, \widetilde{d}({\mathfrak{a}^V}), \nu_{\mathfrak{a}^V}\beta y),\quad\text{and}\\
    &\mathfrak{a} = (\mu_{\mathfrak{a}}\gamma y, \widetilde{d}(\mathfrak{a}), \nu_{\mathfrak{a}}\gamma y).
    \end{split}
\end{equation}
Hence $W_0 := U \cap Z(\mu_{\mathfrak{a}^U}\alpha, \nu_{\mathfrak{a}^U}\alpha)$ is a
compact open neighbourhood of $\mathfrak{a}^U$ and $W_1 := V \cap
Z(\mu_{\mathfrak{a}^V}\beta, \nu_{\mathfrak{a}^V}\beta)$ is a compact open neighbourhood
of $\mathfrak{a}^V$. We have $s(W_0) = r(W_1)$ and $W = W_0W_1$ is a neighbourhood of
$\mathfrak{a}$ contained in $UV \cap Z(\mu_{\mathfrak{a}}\gamma,
\nu_{\mathfrak{a}}\gamma)$. For $\mathfrak{b} \in W$, we have
\[
\mu_{\mathfrak{b}^U} = \mu_{\mathfrak{a}^U},\ \nu_{\mathfrak{b}^U} = \nu_{\mathfrak{a}^U},\ \mu_{\mathfrak{b}^V}
    = \mu_{\mathfrak{a}^V},\ \nu_{\mathfrak{b}^V} = \nu_{\mathfrak{a}^V}, \mu_{\mathfrak{b}} = \mu_{\mathfrak{a}},
        \text{ and }\nu_{\mathfrak{b}} = \nu_{\mathfrak{a}},
\]
and equations~\eqref{eq:alphabetagamma} are satisfied with $\mathfrak{b}$ in place of
$\mathfrak{a}$ and the same paths $\alpha,\beta,\gamma$ but a different $y$. In
particular, by \cite[Lemma~6.3]{KPS4},
\[
\sigma_c(\mathfrak{b}^U, \mathfrak{b}^V) = \sigma_c(\mathfrak{a}^U, \mathfrak{a}^V) =
    c(\mu_{\mathfrak{a}^U}, \alpha)\overline{c(\nu_{\mathfrak{a}^U}, \alpha)}c(\mu_{\mathfrak{b}^V}, \beta)\overline{c(\nu_{\mathfrak{b}^V}, \beta)}
                        c(\nu_{\mathfrak{a}}, \gamma)\overline{c(\mu_{\mathfrak{a}}, \gamma)}.
\]
By compactness of $U, V$ and $UV$, there is a finite collection $\Ww$ of mutually
disjoint compact open bisections such that $UV = \bigcup \Ww$, and for $W \in \Ww$ there
exist $\mu_W$, $\nu_W$, $\eta_W$, $\zeta_W$, $\sigma_W$, $\tau_W$, $\alpha_W$, $\beta_W$,
and $\gamma_W$ in $\Lambda$ such that, for all $\mathfrak{a} \in W$,
\[
(\pi_c(1_U) 1_V)(\mathfrak{a})
    = c(\mu_W, \alpha_W)\overline{c(\nu_W, \alpha_W)}c(\eta_W, \beta_W)\overline{c(\zeta_W, \beta_W)}
                        c(\sigma_{W}, \gamma_W)\overline{c(\tau_{W}, \gamma_W)}.
\]
For each $W \in \Ww$ let $U_W$ and $V_W$ be the subsets of $U$ and $V$ such that $W = U_W
V_W$ (and $s(U_W) = r(V_W)$).  For each $W$,
\begin{equation}
\pi_c(1_{U_W})1_{V_W} = c(\mu_W, \alpha_W)\overline{c(\nu_W, \alpha_W)}c(\eta_W, \beta_W)\overline{c(\zeta_W, \beta_W)}
                        c(\sigma_{W}, \gamma_W)\overline{c(\tau_{W}, \gamma_W)} 1_W.\label{eq:pi_c action}
\end{equation}
Now suppose that $c_n \to c$ in $\Zcat2(\Lambda, \TT)$ and fix $\varepsilon > 0$. Since
multiplication in $\TT$ is continuous, \eqref{eq:pi_c action} implies that there exists
$N$ such that $n \ge N$ implies $\|\pi_{c_n}(1_{U_W})1_{V_W} - \pi_c(1_{U_W})1_{V_W}\| <
\varepsilon/|\Ww|$ for each $W \in \Ww$. Since $\bigcup \Ww$ is a bisection, given any
collection $\{a_W : W \in \Ww\}$ of scalars, we have
\[
\Big\|\sum_{W \in \Ww} a_W 1_W\Big\|^2 = \sup_{x \in \Gg_\Lambda^{(0)}} \sum_{W, W' \in \Ww}
    \overline{a_W} a_{W'} 1_{W^{-1} W'}(x) = \max_{W \in \Ww} |a_W|^2.
\]
That $\pi_c(1_U)1_V = \sum_{W \in \Ww} \pi_c(1_{U_W})1_{V_W}$ gives $\|\pi_{c_n}(1_U)1_V
- \pi_c(1_U)1_V\| < \varepsilon$ for all $n \ge N$.

Fix $f \in \ell^2(s)$ and $\varepsilon > 0$. There is a finite set $\Vv$ of compact open
bisections and a linear combination $f_0 := \sum_{V \in \Vv} a_V 1_V$ such that $\|f -
f_0\| \le \varepsilon/3$. The preceding paragraph yields $N$ such that
$\|\pi_{c_n}(1_U)1_V - \pi_c(1_U)1_V\| < \frac{\varepsilon}{3|\Vv|\,|a_V|}$ for $V \in
\Vv$ and $n \ge N$. For $n \ge N$,
\begin{align*}
\|\pi_{c_n}(1_U)f - \pi_c(1_U)f\|
    \le \|\pi_{c_n}(1_U)f - \pi_{c_n}(1_U)f_0\|
        &{}+ \|\pi_{c_n}(1_U)f_0 - \pi_c(1_U)f_0\| \\
    &\quad{} + \|\pi_c(1_U)f_0 - \pi_c(1_U)f\|
    \le \varepsilon
\end{align*}
since $\|\pi_c(1_U)\| \le \|1_U\| = 1$. Hence $c \mapsto \pi_c(1_U)$ is strongly
continuous.
\end{proof}

For the following corollary, recall that the isomorphism $C^*(\Lambda, c) \cong
C^*(\Gg_\Lambda, \sigma_c)$ of \cite[Theorems 6.7 and 7.9]{KPS4} carries each $s_\lambda
s^*_\mu$ to a finite linear combination, with coefficients in $\TT$, of elements of the
form $1_{Z(\lambda\alpha,\mu\alpha)}$.

\begin{cor}\label{cor:lower semicontinuity}
Resume the hypotheses of Proposition~\ref{prp:strong cty}. If $a : \Lambda *_s \Lambda
\to \CC$ has finite support, then $S_a : c \mapsto \sum_{\lambda,\mu} a(\lambda,\mu)
\pi_c(s_\lambda s^*_\mu)$ is strongly continuous from $\Zcat2(\Lambda, \TT)$ to
$\Ll(\ell^2(s))$, and $c \mapsto \|S_a(c)\|$ is lower-semicontinuous from
$\Zcat2(\Lambda, \TT)$ to $[0,\infty)$.
\end{cor}

\begin{proof}
Fix a finitely supported function $a$, and suppose that $c_n \to c$ in
$\Zcat2(\Lambda,\TT)$. Identify $C^*(\Lambda, c)$ with $C^*(\Gg_\Lambda, \sigma_c)$ as in
\cite[Theorems 6.7 and 7.9]{KPS4}. We may then assume that each $\pi_c(s_\lambda s^*_\mu)
= \sum_{\tau \in F} z_{\tau} 1_{Z(\lambda\tau,\mu\tau)}$ for some finite set $F$ and $z_{\tau} \in \TT$. By
relabelling we may assume that $S_a(c) = \sum_{\lambda,\mu}
\tilde{a}(\lambda,\mu)\pi_c(1_{Z(\lambda,\mu)})$ where $\tilde{a}$ has finite support.
Fix $x \in \ell^2(s)$ of norm 1 and $\varepsilon > 0$. Let $M := \sum_{\lambda,\mu}
|\tilde{a}_{\lambda,\mu}|$. Whenever $\tilde{a}_{\lambda,\mu} \not= 0$,
Proposition~\ref{prp:strong cty} yields $N_{\lambda,\mu} \in \NN$ such that
$\big\|\big(\pi_{c_n}(1_{Z(\lambda,\mu)}) - \pi_c (1_{Z(\lambda,\mu)})\big) x\big\| \le
\frac{\varepsilon}{M}$. Let $N := \max_{\lambda,\mu} N_{\lambda,\mu}$. For $n \ge N$, we
have $\|(S_a(c_n) - S_a(c))x\| \le \frac{\varepsilon}{M} \sum_{\lambda,\mu} |\tilde{a}_{\lambda,\mu}|  \le
\varepsilon$. Thus $S_a$ is strongly continuous.

To see that $c \mapsto \|S_a(c)\|$ is lower-semicontinuous, fix $\varepsilon > 0$. We
must show that there exists $N \in \NN$ such that $\|S_a(c_n)\| \ge \|S_a(c)\| -
\varepsilon$ whenever $n \ge N$. For this, fix $x \in \ell^2(s)$ such that $\|x\| = 1$
and $\|S_a(c) x\| \ge \|S_a(c)\| - \varepsilon/2$. By the preceding paragraph there
exists $N \in \NN$ such that $n \ge N$ implies $\|S_a(c_n) x - S_a(c)x\| \le
\varepsilon/2$, and in particular
\[
    \|S_a(c_n)\| \ge \|S_a(c_n)x\| \ge \|S_a(c) x\| - \varepsilon/2 \ge \|S_a(c)\| - \varepsilon.
\]

The final assertion is clear because each $C^*(\Lambda, c) = \clsp\{s_\lambda s^*_\mu :
\lambda,\mu \in \Lambda\}$.
\end{proof}

\begin{cor}\label{cor:cts field}
Let $\Lambda$ be a row-finite $k$-graph with no sources, let $A$ be an abelian group, and
let $c \in \Zcat2(\Lambda,A)$. The upper-semicontinuous $C^*$-bundle
over $\widehat{A}$ with fibres $\pi_\chi(\csLA{c}) \cong C^*(\Lambda, \chi \circ c)$
described in Proposition~\ref{prp:extension algebra bundle} is continuous.
\end{cor}
\begin{proof}
The map $\chi \mapsto \chi \circ c$ from $\widehat{A}$ to $\Zcat2(\Lambda,\TT)$ is
continuous. Let $\{s_\lambda : \lambda \in \Lambda\}$ be the generators of $C^*(\Lambda,
\chi \circ c)$. Proposition~\ref{prp:strong cty} implies that for a finitely supported
function $a : \Lambda *_s \Lambda \to \CC$, we have $\|\sum a(\lambda,\mu) s_\lambda
s^*_\mu\| = \|S_a(\chi \circ c)\|$. It therefore follows from Corollary~\ref{cor:lower
semicontinuity} that $\chi \mapsto \big\|\pi_\chi\big(\sum a(\lambda,\mu)
i_\Lambda(\lambda) i_\Lambda(\mu)^*\big)\big\|$ is lower semicontinuous. Since
Proposition~\ref{prp:extension algebra bundle} implies that it is also
upper-semicontinuous, the result follows.
\end{proof}

\begin{example}[Fields of noncommutative tori]\label{eg:nct}
Fix $k \ge 2$. Let $T_k$ be a copy of $\NN^k$ regarded as a $k$-graph with degree functor
the identity map on $\NN^k$. Let $A$ be the free abelian group generated by $\{(i,j) : 1
\le j < i \le k\}$. There is an $A$-valued 2-cocycle $c$ on $T_k$ given by
\[\textstyle
c(m,n) := \sum_{j < i} m_in_j \cdot (i, j).
\]
Using relations (\ref{it:LArep1})--(\ref{it:LArep5}), one can check that $C^*(T_k, A, c)$
is universal for unitaries $\{U_m : m \in \ZZ^k\} \cup \{W_a : a \in A\}$ such that, for
$m, n \in T_k$ and $a,b \in A$,
\[
W_a W_b = W_{a+b},\qquad U_m U_n = W_{c(m,n)} U_{m+n}, \qquad\text{and}\qquad U_mW_a = W_aU_m.
\]
Hence $U_{e_i} U_{e_j} = W_{(i,j)} U_{e_i + e_j} = W_{(i,j)} U_{e_j} U_{e_i}$ for $j <
i$. Let $G$ be the group generated by $\{g_1, \dots, g_k\} \cup \{h_{i, j} : 1 \le j < i
\le k \}$ such that $g_ig_j = h_{i, j}g_jg_i$ and $g_\ell h_{i, j} = h_{i, j}g_\ell$ for
all $i, j, \ell$ with $j < i$. Then $C^*(T_k, A, c)$ and $C^*(G)$ have the same universal
property, so coincide.

Fix scalars $z_{i,j} \in \TT$ for $1 \le j < i \le k$. The noncommutative torus $A_z$ is
the universal $C^*$-algebra generated by unitaries $U_1, \dots, U_k$ subject to $U_iU_j =
z_{i,j} U_jU_i$ for $j < i$ (see, for example, \cite[1.4]{Elliott:Neptun84}). There is a
unique character $\chi_z$ of $A$ satisfying $\chi_z(i, j) = z_{i,j}$ for $j < i$. The
twisted $k$-graph $C^*$-algebra $C^*(T_k, \chi_z \circ c)$ is the quotient of $C^*(T_k,
A, c)$ by the ideal generated by $\{W_{(i, j)} - z_{i,j} 1 : j < i\}$. Hence the
preceding paragraph implies that $C^*(T_k, \chi_z \circ c)$ is the noncommutative torus
$A_z$. Thus Corollary~\ref{cor:cts field} shows that $C^*(T_k , A , c)$ is the algebra of
sections of a continuous bundle over $\prod_{j < i \le k} \TT$ whose fibre over $z$ is
$A_z$.

When $k = 2$, the group $G$ is the integer Heisenberg group $H_3(\ZZ)$, and we recover
the description of $C^*(H_3(\ZZ))$ as the sections of a continuous field of rotation
algebras \cite{AndersonPaschke:HJM89}.
\end{example}

\begin{example}[Fields of Heegaard-type quantum spheres]\label{eg:Heegaard bundle}
Consider the Heegaard-type quantum spheres $C(S^3_{00\theta})$ of \cite{BHMS}. It was
shown in \cite[Example~7.10]{KPS3} that these $C^*$-algebras are the twisted
$C^*$-algebras of a finite $2$-graph $\Lambda$, and that $\Hcat2(\Lambda, \TT) = \TT$,
with the cohomology class corresponding to $e^{2\pi \imath \theta}$ represented by
$c_\theta : (\mu,\nu) \mapsto e^{2\pi \imath(d(\mu)_2 d(\nu)_1)\theta}$. We then have
$C(S^3_{00\theta}) \cong C^*(\Lambda, c_\theta)$.

Let $b \in \Zcat2(\Lambda, \ZZ)$ be the cocycle $b(\mu,\nu) = d(\mu)_2 d(\nu)_1$.
Identifying $\TT$ with $\widehat{\ZZ}$ by $z \mapsto (\chi_z : n \mapsto z^n)$, each
$c_\theta$ is equal to $\chi_{e^{2\pi\imath\theta}} \circ b$. Thus Corollary~\ref{cor:cts
field} shows that $\csLA{c_\theta}$ is the section algebra of a continuous field over
$\TT$ of Heegaard-type quantum spheres.
\end{example}

\section{A field of AF algebras}\label{sec:AFbundle}
If $\Lambda$ is a $k$-graph and $b$ is a function from $\Lambda^0$ to $A$, then the
corresponding $1$-coboundary $\dcat0b : \Lambda \to A$ is given by $\dcat0b(\lambda) =
b(s(\lambda)) - b(r(\lambda))$.

Suppose that the degree map on $\Lambda$ satisfies $d = \dcat{0}b$ for some $b :
\Lambda^0 \to \ZZ^k$. By \cite[Lemma 5.4]{KP2000} it follows that $C^*(\Lambda)$ is AF;
and \cite[Theorem 8.4]{KPS4} says that for any $c \in \Zcat2(\Lambda, \TT)$ the twisted
algebra $C^*(\Lambda, c)$ is isomorphic to $C^*(\Lambda)$. In Theorem~\ref{thm:dirlim
tensor} we combine these results with the bundle structure of $C^*(\Lambda, A, c)$ given
in Proposition~\ref{prp:extension algebra bundle} to show that $\csLA{c} \cong C^*(A)
\otimes C^*(\Lambda)$ for every $c \in \Zcat2(\Lambda, A)$.

Given a nonempty set $X$, we write $\Kk_X$ for the unique nonzero $C^*$-algebra generated
by elements $\{\theta_{x,y} : x,y \in X\}$ satisfying $\theta_{x,y}^* = \theta_{y,x}$ and
$\theta_{x,y}\theta_{w,z} = \delta_{y,w} \theta_{x,z}$. We call the $\theta_{x,y}$
\emph{matrix units} for $\Kk_X$.

\begin{lem}\label{lem:compact bundle}
Let $\Lambda$ be a row-finite $k$-graph with no sources and suppose that $d = \dcat{0}b$
for some $b : \Lambda^0\to \ZZ^k$. Let $A$ be an abelian group, and fix $c \in
\Zcat2(\Lambda, A)$. For each $n \in \ZZ^k$, the subspace $B_n :=
\clsp\{i_\Lambda(\lambda) i_A(f) i_\Lambda(\mu)^* : s(\lambda) = s(\mu) \in b^{-1}(n), f
\in C^*(A)\}$ of $\csLA{c}$ is a $C^*$-subalgebra. The formula $f \otimes
\theta_{\lambda, \mu} \mapsto i_\Lambda(\lambda) i_A(f) i_\Lambda(\mu)^*$ determines an
isomorphism $\rho_n : C^*(A) \otimes \big(\bigoplus_{b(v) = n} \Kk_{\Lambda v}\big) \to
B_n$.
\end{lem}

\begin{proof}
As in the proof of \cite[Lemma~5.4]{KP2000} (or \cite[Lemma~8.4]{KPS4}),
$\{i_\Lambda(\lambda) i_\Lambda(\mu)^* : s(\lambda) = s(\mu) \in b^{-1}(n)\}$ is a set of
nonzero matrix units, and so spans a copy of $\bigoplus_{b(v) = n} \Kk_{\Lambda v}$ in
$\Mm(\csLA{c})$. Each $i_\Lambda(\lambda) i_\Lambda(\mu)^*$ commutes with $i_A(C^*(A))$,
and so $B_n$ is a $C_0(\widehat{A}\,)$-algebra. The universal property of the tensor
product gives a (clearly surjective) homomorphism $\rho_n : C^*(A) \otimes
\big(\bigoplus_{b(v) = n} \Kk_{\Lambda v}\big) \to B_n$ such that $\rho_n(f \otimes
\theta_{\lambda,\mu}) = i_\Lambda(\lambda) i_A(f) i_{\Lambda}(\mu)^*$.

If $\chi \in \widehat{A}$ and $f \in C^*(A)$ satisfy $f(\chi) = 1$, then $\pi_\chi$
carries the $i_\Lambda(\lambda) i_A(f) i_\Lambda(\mu)^*$ to nonzero matrix units in
$C^*(\Lambda, \chi \circ c)$. Hence evaluation at $\chi$ induces an isomorphism
$\big(C^*(A) \otimes \big(\bigoplus_{b(v) = n} \Kk_{\Lambda v}\big)\big)_\chi \cong
(B_n)_\chi$. Thus $\rho_n$ is an isomorphism.
\end{proof}

The idea of the following proof is due to Ben Whitehead \cite{Whiteheadhonours}.

\begin{thm}\label{thm:dirlim tensor}
Let $\Lambda$ be a row-finite $k$-graph with no sources and suppose that $d
= \dcat{0}b$ for some $b : \Lambda^0\to \ZZ^k$. Let $A$ be an abelian group, and fix $c
\in \Zcat2(\Lambda, A)$. For $m \le n \in \NN^k$, the inclusion $j_{m,n} : B_m
\subseteq B_n$ is given by
\[
j_{m,n}(i_\Lambda(\lambda) i_A(f) i_\Lambda(\mu)^*)
    = \sum_{\nu \in s(\lambda)\Lambda^{n - m}} i_A(c(\lambda,\nu) - c(\mu,\nu))
        i_\Lambda(\lambda\nu) i_A(f) i_\Lambda(\mu\nu)^*.
\]
Furthermore, $\csLA{c} = \varinjlim(B_n, j_{m,n}) \cong C^*(A) \otimes C^*(\Lambda)$.
\end{thm}

\begin{proof}
The formula for the inclusion maps $j_{m,n}$ follows from~(\ref{it:LArep3}) and
(\ref{it:LArep5}). Each spanning element of $\csLA{c}$ belongs to $\bigcup_n B_n$, and
hence $\csLA{c} = \varinjlim(B_n, j_{m,n})$. For $n \in \NN^k$, let $\rho_n : C^*(A)
\otimes \big(\bigoplus_{b(v) = n} \Kk_{\Lambda v}\big) \to B_n$ be as in
Lemma~\ref{lem:compact bundle}. For $m \le n \in \NN^k$ let $\iota_{m,n} :
\bigoplus_{b(v) = m} \Kk_{\Lambda v} \to \bigoplus_{b(w) = n} \Kk_{\Lambda w}$ be the
inclusion $\iota_{m,n}(\theta_{\lambda,\mu}) = \sum_{\alpha \in s(\lambda)\Lambda^{n-m}}
\theta_{\lambda\alpha, \mu\alpha}$.

Let $\mathbf{1} := (1, \dots, 1) \in \NN^k$. Define $\kappa : \Lambda \to A$ recursively
by
\[
\kappa(\mu) =
\begin{cases}
0 & \text{if $d(\mu) \not\ge \mathbf{1}$,} \\
\kappa(\lambda) + c(\lambda, \alpha) & \text{if  } \mu = \lambda\alpha \text{ and } d(\alpha) = \mathbf{1}.
\end{cases}
\]

Let $u_A : A \to \Uu\Mm(C^*(A))$ denote the canonical map. For $h \in \NN$, an
$\varepsilon/3$-argument (see, for example, \cite{MaloneyPaskEtAl:xx11}) shows that
$\sum_{b(s(\lambda)) = h\cdot\mathbf{1}} u_A(\kappa(\lambda)) \otimes
\theta_{\lambda,\lambda}$ converges strictly to a unitary $U_h \in \Uu\Mm\big(C^*(A)
\otimes \big(\bigoplus_{b(v) = h\cdot\mathbf{1}} \Kk_{\Lambda v}\big)\big)$.

Fix $h \in \NN$ and $\lambda,\mu \in \Lambda$ with $s(\lambda) = s(\mu) \in
b^{-1}(h\cdot\mathbf{1})$. Then
\begin{align*}
j_{h\cdot\mathbf{1}, (h+1)\cdot\mathbf{1}} &\circ \rho_{h\cdot\mathbf{1}}
        (U_h (f \otimes \theta_{\lambda,\mu}) U^*_h) \\
    &=  j_{h\cdot\mathbf{1}, (h+1)\cdot\mathbf{1}}
    \big( i_A(\kappa(\lambda)-\kappa(\mu))i_\Lambda(\lambda) i_A(f) i_\Lambda(\mu)^*\big) \\
    &= \textstyle\sum_{\alpha \in s(\lambda)\Lambda^{\mathbf{1}}}
       i_A\big((\kappa(\lambda)+c(\lambda,\alpha)) - (\kappa(\mu)+c(\mu,\alpha))\big)
            i_\Lambda(\lambda\alpha) i_A(f) i_\Lambda(\mu\alpha)^*\\
    &= \textstyle\sum_{\alpha \in s(\lambda)\Lambda^{\mathbf{1}}}
        \rho_{(h+1)\cdot \mathbf{1}}\big(u_A\big(\kappa(\lambda\alpha) - \kappa(\mu\alpha)\big)
                f \otimes \theta_{\lambda\alpha,\mu\alpha}\big)\\
    &= \rho_{(h+1)\cdot \mathbf{1}} \big(U_{h+1}
        (f \otimes \iota_{h\cdot\mathbf{1}, (h+1)\cdot\mathbf{1}}(\theta_{\lambda,\mu}))U_{h+1}^*\big).
\end{align*}
That is, $j_{h\cdot\mathbf{1},(h+1)\cdot\mathbf{1}} \big(\rho_{h\cdot\mathbf{1}}\circ
\Ad(U_h)\big) = \big(\rho_{(h+1)\cdot \mathbf{1}} \circ \Ad(U_{h+1})\big)(\id_{C^*(A)}
\otimes \iota_{h\cdot\mathbf{1}, (h+1)\cdot\mathbf{1}})$. Thus
\begin{align*}
\csLA{c} &= \varinjlim(B_n, j_{m,n})
\textstyle\cong \varinjlim(B_{h \cdot \mathbf{1}}, j_{h\cdot \mathbf{1},(h+1)\cdot \mathbf{1}}) \\
    &\textstyle\cong \varinjlim\big(C^*(A) \otimes \big(\bigoplus_{b(v) = h\cdot\mathbf{1}} \Kk_{\Lambda v}\big),
                          \id_{C^*(A)} \otimes \iota_{h\cdot\mathbf{1}, (h+1)\cdot\mathbf{1}}\big).
\end{align*}
Since $C^*(\Lambda) \cong \varinjlim\big(\bigoplus_{b(v) = h\cdot \mathbf{1}}
\Kk_{\Lambda v}, \iota_{h\cdot \mathbf{1},(h+1)\cdot \mathbf{1}})$, the result follows.
\end{proof}

\begin{cor}\label{cor:[0,1]-bundle}
Let $\Lambda$ be a row-finite $k$-graph with no sources. Suppose that $d = \dcat0b$ for
some $b : \Lambda^0 \to \ZZ^k$. Fix $c \in \Zcat2(\Lambda, \RR)$, and identify
$\widehat{\RR}$ with $\RR$ via $t \mapsto (\omega_t : s \mapsto e^{\imath ts})$. Then
$C^*(\Lambda, \RR, c)|_{[0,1]} \cong C([0,1]) \otimes C^*(\Lambda)$, and the maps $\pi_t
: C^*(\Lambda, \RR, c) \to C^*(\Lambda, \omega_t \circ c)$ induce isomorphisms
$K_*(C^*(\Lambda, \RR, c)|_{[0,1]}) \cong K_*(C^*(\Lambda))$ such that
$(\pi_t)_*\big([i_\Lambda(v)]_0\big) = [s_v]_0$ for $v \in \Lambda^0$.
\end{cor}

\begin{proof}
The first statement is a special case of Theorem~\ref{thm:dirlim tensor}. The K\"unneth
theorem \cite[Theorem~23.1.3]{Blackadar:K-theory} and that $K_0(C([0,1]), [1]) \cong
(\ZZ, 1)$ gives the second statement.
\end{proof}

\section{\texorpdfstring{$K$}{K}-theory}\label{sec:K-theory}

In this section we prove our main result, Theorem~\ref{thm:main K-theory}: if $c \in
\Zcat2(\Lambda, \TT)$ arises by exponentiation of an $\RR$-valued $2$-cocycle, then
$K_*(C^*(\Lambda, c)) \cong K_*(C^*(\Lambda))$. Our argument is essentially that devised
by Elliott to prove \cite[Theorem~2.2]{Elliott:Neptun84}. We believe that Gwion Evans'
argument \cite{Evans:NYJM08} could also be adapted to obtain a spectral sequence
converging to $K_*(C^*(\Lambda, c))$. The advantage of our approach is that when $c$ is
of exponential form we reduce the problem of computing $K_*(C^*(\Lambda, c))$ to that of
computing $K_*(C^*(\Lambda))$ whether or not the latter is computable using Evans'
spectral sequence.

Given an action $\delta$ of a discrete group $G$ on a $C^*$-algebra $D$,  we write
$(\jmath_D, \jmath_G)$ for the universal covariant representation of $(D, G, \delta)$.
That is, $\jmath_D : D \to D \times_\delta G$ (we also write $\jmath_D : \Mm(D)\to \Mm (D
\times_\delta G)$ for the extension) and $\jmath_G : G \to \Uu\Mm(D \times_\delta G)$ are
the canonical embeddings. The following is surely well known to $K$-theory experts and is
implicitly contained in the proof of \cite[Theorem~2.2]{Elliott:Neptun84}.

\begin{thm}\label{thm:elliott}
Let $\phi : B \to C$ be a homomorphism of $C^*$-algebras which induces an isomorphism
$\phi_* : K_*(B) \to K_*(C)$. Suppose that $\beta : \ZZ^k \to \Aut(B)$ and $\gamma :
\ZZ^k \to \Aut(C)$ are actions such that $\gamma_n \circ \phi = \phi \circ \beta_n$ for
all $n \in \ZZ^k$. There is a homomorphism $\widetilde{\phi} : B \times_\beta \ZZ^k \to C
\times_\gamma \ZZ^k$ such that $\widetilde{\phi}(\jmath_B(b) \jmath_{\ZZ^k}(n)) =
\jmath_C(\phi(b))\jmath_{\ZZ^k}(n)$. The induced map $\widetilde{\phi}_* : K_*(B \times_\beta
\ZZ^k) \to K_*(C \times_\gamma \ZZ^k)$ is an isomorphism, and $\widetilde{\phi}_* \circ
(\jmath_B)_* = (\jmath_C)_* \circ \phi_*$.
\end{thm}

\begin{proof}
Universality of the crossed product applied to the covariant representation $(\phi \circ
\jmath_B, \jmath_{\ZZ^k})$ yields a homomorphism $\widetilde{\phi} : B \times_\beta \ZZ^k \to C
\times_\gamma \ZZ^k$ as described.

To see that $\widetilde{\phi}$ induces an isomorphism in $K$-theory satisfying
$\widetilde{\phi}_* \circ (\jmath_B)_* = (\jmath_C)_* \circ \phi_*$, we proceed by induction on
$k$.

The base case $k = 0$ is trivial. Fix $n \ge 0$, suppose as an inductive hypothesis that
the result holds for all $k \le n$, and consider $k = n+1$. Let $\beta^1 : \ZZ \to
\Aut(B)$ be the restriction of $\beta$ to $\{(m, 0, \dots, 0) : m \in \ZZ\} \subseteq
\ZZ^k$ and let $\beta^n$ be the $\ZZ^n$ action on $B \times_{\beta^1} \ZZ$ induced by the
restriction of $\beta$ to the last $n$ coordinates of $\ZZ^k$. Universality implies that
$B \times_\beta \ZZ^k$ is canonically isomorphic to $(B \times_{\beta^1} \ZZ)
\times_{\beta^n} \ZZ^n$. Define $\gamma^1 : \ZZ \to \Aut(C)$ and $\gamma^n : \ZZ^n \to
\Aut(C \times_{\gamma^1} \ZZ)$ similarly; by hypothesis, $\phi$ is equivariant for
$\beta^1$ and $\gamma^1$, and so it induces a homomorphism $\widetilde{\phi}^1 : B
\times_{\beta^1} \ZZ \to C \times_{\gamma^1} \ZZ$. Naturality of the Pimsner-Voiculescu
sequence induces a commuting diagram
\[
\begin{tikzpicture}[scale=0.8]
    \node (13) at (1,3) {$K_0(C)$};
    \node (53) at (5,3) {$K_0(C)$};
    \node (93) at (9,3) {$K_0(C \times_{\gamma^1} \ZZ)$};
    \node (91) at (9,1) {$K_1(C)$};
    \node (51) at (5,1) {$K_1(C)$};
    \node (11) at (1,1) {$K_1(C \times_{\gamma^1} \ZZ)$};
    \draw[-stealth] (13)--(53) node[pos=0.5, above] {\small$1 - \gamma^1_*$};
    \draw[-stealth] (53)--(93) node[pos=0.5, above] {\small$(\jmath_C)_*$};
    \draw[-stealth] (93)--(91);
    \draw[-stealth] (91)--(51) node[pos=0.5, above] {\small$1 - \gamma^1_*$};
    \draw[-stealth] (51)--(11) node[pos=0.5, above] {\small$(\jmath_C)_*$};
    \draw[-stealth] (11)--(13);
    \node (04) at (-1,5) {$K_0(B)$};
    \node (54) at (5,5) {$K_0(B)$};
    \node (A4) at (11,5) {$K_0(B \times_{\beta^1} \ZZ)$};
    \node (A0) at (11,-1) {$K_1(B)$};
    \node (50) at (5,-1) {$K_1(B)$};
    \node (00) at (-1,-1) {$K_1(B \times_{\beta^1} \ZZ)$};
    \draw[-stealth] (04)--(54) node[pos=0.5, above] {\small$1 - \beta^1_*$};
    \draw[-stealth] (54)--(A4) node[pos=0.5, above] {\small$(\jmath_B)_*$};
    \draw[-stealth] (A4)--(A0);
    \draw[-stealth] (A0)--(50) node[pos=0.5, above] {\small$1 - \beta^1_*$};
    \draw[-stealth] (50)--(00) node[pos=0.5, above] {\small$(\jmath_B)_*$};
    \draw[-stealth] (00)--(04);
    \draw[-stealth] (04)--(13)  node[pos=0.6, anchor=south west, inner sep=0pt] {\small$\phi_*$};
    \draw[-stealth] (54)--(53)  node[pos=0.5, right] {\small$\phi_*$};
    \draw[-stealth] (A4)--(93)  node[pos=0.7, anchor=south east, inner sep=0pt] {\small$\widetilde{\phi}^1_*$};
    \draw[-stealth] (A0)--(91)  node[pos=0.6, anchor=north east, inner sep=0pt] {\small$\phi_*$};
    \draw[-stealth] (50)--(51)  node[pos=0.5, left] {\small$\phi_*$};
    \draw[-stealth] (00)--(11)  node[pos=0.7, anchor=north west, inner sep=0pt] {\small$\widetilde{\phi}^1_*$};
\end{tikzpicture}
\]
in which the $\phi_*$ are isomorphisms. The Five Lemma implies that the
$\widetilde{\phi}^1_*$ are isomorphisms. The relations $\widetilde{\phi}^1_* \circ
(\jmath_B)_* = (\jmath_C)_* \circ \phi_*$ appear as squares in the diagram.

The inductive hypothesis applied to the actions $\beta^n$ and $\gamma^n$ and the
equivariant homomorphism $\widetilde{\phi}^1$ yields a homomorphism $\widetilde{\phi}^n :
(B \times_{\beta^1} \ZZ) \times_{\beta^n} \ZZ^n \to (C \times_{\gamma^1} \ZZ)
\times_{\gamma^n} \ZZ^n$ which induces an isomorphism in $K$-theory and satisfies
$\widetilde{\phi}^n_* \circ (\jmath_{B \times_{\beta^1} \ZZ})_* = (\jmath_{C \times_{\gamma^1}
\ZZ})_* \circ \widetilde{\phi}^1_*$. Since $\phi$ intertwines the isomorphism $(B
\times_{\beta^1} \ZZ) \times_{\beta^n} \ZZ^n \cong B \times_\beta \ZZ^{n+1}$ and the
corresponding isomorphism for $C, \gamma$, the squares in the following diagram (in which
we have suppressed the subscripts on the inclusion maps $\jmath$) commute.
\[
\begin{tikzpicture}[yscale=0.8]
    \node (00) at (0,0) {$K_*(C)$};
    \node (40) at (3.5,0) {$K_*(C \times_{\gamma^1} \ZZ)$};
    \node (80) at (8,0) {$K_*((C \times_{\gamma^1} \ZZ) \times_{\gamma^n} \ZZ^n)$};
    \node (C0) at (12.5,0) {$K_*(C \times_{\gamma} \ZZ^k)$};
    \draw[-stealth] (00)--(40) node[pos=0.5, above] {\small$\jmath_*$};
    \draw[-stealth] (40)--(80) node[pos=0.5, above] {\small$\jmath_*$};
    \draw[-stealth] (80)--(C0) node[pos=0.5, above] {\small$\cong$};
    \node (02) at (0,2) {$K_*(B)$};
    \node (42) at (3.5,2) {$K_*(B \times_{\beta^1} \ZZ)$};
    \node (82) at (8,2) {$K_*((B \times_{\beta^1} \ZZ) \times_{\beta^n} \ZZ^n)$};
    \node (C2) at (12.5,2) {$K_*(B \times_{\beta} \ZZ^k)$};
    \draw[-stealth] (02)--(42) node[pos=0.5, above] {\small$\jmath_*$};
    \draw[-stealth] (42)--(82) node[pos=0.5, above] {\small$\jmath_*$};
    \draw[-stealth] (82)--(C2) node[pos=0.5, above] {\small$\cong$};
    \draw[-stealth] (02)--(00)  node[pos=0.5, right] {\small$\phi_*$};
    \draw[-stealth] (42)--(40)  node[pos=0.5, right] {\small$\widetilde{\phi}^1_*$};
    \draw[-stealth] (82)--(80)  node[pos=0.5, right] {\small$\widetilde{\phi}^n_*$};
    \draw[-stealth] (C2)--(C0)  node[pos=0.5, right] {\small$\widetilde{\phi}_*$};
\end{tikzpicture}
\]
Commutativity of the outside rectangle gives $\widetilde{\phi}_* \circ (\jmath_B)_* = (\jmath_C)_*
\circ \phi_*$.
\end{proof}

The Morita equivalence described in the following lemma is an application of Takai
duality, and could likely be deduced from arguments like those of
\cite[Section~5]{KP2000}. However, the details of the isomorphism~\eqref{eq:corner
isomorphism} implementing the Morita equivalence will be important in the proof of our
main theorem. Recall from \cite{KP2000} that if $\Lambda$ is a $k$-graph, then the
skew-product $\Lambda \times_d \ZZ^k$ of $\Lambda$ by its degree functor is the set
$\Lambda \times \ZZ^k$ with degree map $d(\lambda, n) = d(\lambda)$, range and source
maps $r(\lambda, n) = (r(\lambda), n)$ and $s(\lambda, n) = (s(\lambda), n + d(\lambda))$
and composition $(\lambda, n)(\mu, n + d(\lambda)) = (\lambda\mu, n)$. This skew-product
is a $k$-graph, and is row-finite and has no sources if and only if $\Lambda$ has the
same properties.

\begin{lem}\label{lem:not written yet}
Let $\Lambda$ be a row-finite $k$-graph with no sources, and let $c \in \Zcat2(\Lambda,
\TT)$. Define $\phi : \Lambda \times_d \ZZ^k \to \Lambda$ by $\phi(\lambda, n) =
\lambda$. Then $c\circ \phi \in \Zcat2(\Lambda \times_d \ZZ^k, \TT)$. There is an action
$\lt : \ZZ^k \to \Aut(C^*(\Lambda \times_d \ZZ^k, c \circ \phi))$ determined by
$\lt_n(s_{(\lambda, m)}) = s_{(\lambda, m+n)}$. The series
\[\textstyle
\sum_{v \in \Lambda^0} \jmath_{C^*(\Lambda \times_d \ZZ^k, c \circ \phi)}(s_{(v,0)})
\]
converges strictly to a full projection $P_0 \in \Mm \big(C^*(\Lambda \times_d \ZZ^k, c
\circ \phi) \times_{\lt} \ZZ^k\big)$, and there is an isomorphism
\begin{equation}\label{eq:corner isomorphism}
C^*(\Lambda, c) \cong
    P_0 \big(C^*(\Lambda \times_d \ZZ^k, c \circ \phi) \times_{\lt} \ZZ^k\big) P_0,
\end{equation}
which carries each $s_v$ to $\jmath_{C^*(\Lambda \times_d \ZZ^k, c \circ \phi)}(s_{(v,0)})$.
\end{lem}

\begin{proof}
For $n \in \NN^k$, $(\lambda, m) \mapsto s_{(\lambda, m+n)}$ determines a Cuntz-Krieger
$(\Lambda \times_d \ZZ^k, c \circ \phi)$-family in $C^* (\Lambda \times_d \ZZ^k, c \circ
\phi)$. Thus the universal property gives a homomorphism $\lt_n : C^*(\Lambda \times_d
\ZZ^k, c \circ \phi) \to C^*(\Lambda \times_d \ZZ^k, c \circ \phi)$ satisfying
$\lt_n(s_{(\lambda, m)}) = s_{(\lambda, m+n)}$. It is straightforward to check that this
determines an action  $\lt$ of $\ZZ^k$ on $C^*(\Lambda \times_d \ZZ^k, c \circ \phi)$.

To simplify calculations, we identify the generators of $C^*(\Lambda \times_{d} \ZZ^k, c
\circ \phi)$ with their images in $C^*(\Lambda \times_d \ZZ^k, c \circ \phi) \times_{\lt}
\ZZ^k$. For $m \in \ZZ^k$ we write $u_m$ for $\jmath_{\ZZ^k}(m) \in \Uu\Mm\big(C^*(\Lambda
\times_d \ZZ^k, c \circ \phi) \times_{\lt} \ZZ^k\big)$.

It is standard that the sum defining $P_0$ converges strictly to a multiplier projection
(see, for example, \cite[Lemma~2.10]{CBMSbook}). The series $\sum_{(v,n) \in \Lambda^0
\times \ZZ^k} s_{(v,n)}$ is an approximate identity for $C^*(\Lambda \times_d \ZZ^k, c
\circ \phi) \times_{\lt} \ZZ^k$. Hence $P_0$ is full because
\[
s_{(v,n)}
    = \lt_n(s_{(v,0)})
    = u_n s_{(v,0)} u^*_n
    = u_n P_0 s_{(v,0)} u^*_n\quad\text{ for all $(v, n) \in \Lambda^0 \times \ZZ^k$.}
\]

For $\lambda \in \Lambda$, define $t_\lambda := s_{(\lambda, 0)} u_{d(\lambda)}$. Then
$t_\lambda t^*_\lambda = s_{(\lambda, 0)} s^*_{(\lambda, 0)} \le P_0$, and
\[
t^*_\lambda t_\lambda = u_{d(\lambda)}^* s_{(s(\lambda), d(\lambda))} u_{d(\lambda)} = s_{(s(\lambda), 0)} \le P_0,
\]
so the $t_\lambda$ belong to the corner determined by $P_0$.

Straightforward calculations using the relation $u_n s_{(\lambda,m)} u^*_n = s_{(\lambda,
m+n)}$ show that the $t_\lambda$ form a Cuntz-Krieger $(\Lambda, c)$-family in $P_0
(C^*(\Lambda \times_d \ZZ^k, c \circ \phi) \times_{\lt} \ZZ^k) P_0$. Hence the universal
property of $C^*(\Lambda, c)$ gives a homomorphism $\pi_t : C^*(\Lambda, c) \to P_0
\big(C^*(\Lambda \times_d \ZZ^k, c \circ \phi) \times_{\lt} \ZZ^k\big) P_0$ which carries
each $s_v$ to $t_v = \jmath_{C^*(\Lambda \times_d \ZZ^k, c \circ \phi)}(s_{(v,0)})$.

Since the generators of $C^*(\Lambda \times_d \ZZ^k, c \circ \phi)$ are nonzero, the
$t_v$ are nonzero. The gauge action on $C^*(\Lambda \times_d \ZZ^k, c \circ \phi)$
induces an action $\beta$ of $\TT^k$ on $C^*(\Lambda \times_d \ZZ^k, c \circ \phi)
\times_{\lt} \ZZ^k$ satisfying $\beta_z(t_\lambda) = z^{d(\lambda)} t_\lambda$. The
gauge-invariant uniqueness theorem \cite[Corollary~7.7]{KPS4} therefore implies that
$\pi_t$ is injective. To check that $\pi_t$ is surjective, observe that $C^*(\Lambda
\times_d \ZZ^k, c \circ \phi) \times_{\lt} \ZZ^k$ is densely spanned by elements of the
form $s_{(\lambda, m)} s^*_{(\mu,n)} u_p$, where $s(\lambda,m) = s(\mu,n)$ and hence $n-m
= d(\lambda) - d (\mu)$. For a nonzero spanning element $s_{(\lambda, m)} s^*_{(\mu,n)}
u_p$,
\begin{align*}
P_0 s_{(\lambda, m)} s^*_{(\mu,n)} u_p P_0
    &= P_0 s_{(\lambda, m)} u_p s^*_{(\mu,n - p)} P_0
    = \begin{cases}
        s_{(\lambda, 0)} u_{n-m} s^*_{(\mu, 0)} &\text{ if $m = n-p = 0$}\\
        0 &\text{ otherwise.}
    \end{cases}
\end{align*}
Since $n - m = d(\lambda) - d(\mu)$ we have $s_{(\lambda, 0)} u_{n-m} s^*_{(\mu, 0)} =
t_\lambda t^*_\mu$. Thus $\pi_t$ is surjective.
\end{proof}

To prove our main theorem, we need a standard technical lemma, which we include for
completeness and to establish notation.

\begin{lem}\label{lem:action field}
Let $X$ be a compact Hausdorff space, and suppose that $B$ is a $C(X)$-algebra with
respect to a homomorphism $p : C(X) \to \Zz \Mm(B)$. Let $\alpha$ be an action of a
discrete group $G$ on $B$ whose extension to $\Mm(B)$ fixes $p(C(X))$. Then
$\jmath_B\circ p$ maps $C(X)$ into $\Zz \Mm(B \times_\alpha G)$ and $B \times_\alpha G$
is also a $C(X)$-algebra. For each $x \in X$, $\alpha$ induces an action $\alpha^x$ of
$G$ on $B_x$ and the quotient map $\pi^x : B \to B_x$ is equivariant for $\alpha$ and
$\alpha^x$. The induced homomorphism $\widetilde{\pi}^x : B \times_\alpha G \to B_x
\times_{\alpha^x} G$ satisfies
\begin{equation}\label{eq:intertwining}
\widetilde{\pi}^x \circ \jmath_B = \jmath_{B_x} \circ \pi^x
    \quad\text{ and }\quad
\widetilde{\pi}^x \circ \jmath_G = \jmath^x_G,
\end{equation}
where $(\jmath_{B_x}, \jmath^x_G)$ is the universal covariant representation of
$(B_x, G, \alpha^x)$. Moreover, $\widetilde{\pi}^x$ descends to an isomorphism
$(B \times_\alpha G)_x \cong B_x \times_{\alpha^x} G$.
\end{lem}
\begin{proof}
The range of $\jmath_B \circ p$ is central because $\alpha$ fixes $p(C(X))$. Fix $x \in
X$. The quotient map $\pi^x$ is equivariant for $\alpha$ and $\alpha^x$ by definition of
the latter. Hence $(\jmath_{B_x} \circ \pi^x, \jmath^x_G)$ is a covariant representation
of $(B, \alpha, G)$, and so induces a homomorphism $\widetilde{\pi}^x : B \times_\alpha G
\to B_x \times_{\alpha^x} G$ satisfying~\eqref{eq:intertwining}. This $\widetilde{\pi}^x$
descends to a homomorphism $(B \times_\alpha G)_x \to B_x \times_{\alpha^x} G$ because
its kernel contains $\{f \in C(X) : f(x) = 0\}$. Let $\theta^x : B\times_\alpha G \to
(B\times_\alpha G)_x$ be the quotient map, and continue to write $\theta^x$ for the
extension to multiplier algebras. Then $\theta^x \circ \jmath_B$ descends to a
homomorphism $\widetilde{\theta}^x : B_x \to (B\times_\alpha G)_x$, and the pair
$(\widetilde{\theta}^x, \theta^x \circ \jmath_G)$ is a covariant representation of $(B_x,
\alpha^x, G)$ in $(B\times_\alpha G)_x$ whose integrated form is inverse to
$\widetilde{\pi}^x$.
\end{proof}

\begin{thm}\label{thm:main K-theory}
Let $\Lambda$ be a row-finite $k$-graph with no sources. Let $c \in \Zcat2(\Lambda,
\RR)$. For $t \in \RR$ define $\omega_t \in \widehat{\RR}$ by $\omega_t(s) = e^{\imath ts}$;
then $C^*(\Lambda, \omega_t \circ c)$ is unital if and only if $\Lambda^0$ is finite.
There is an isomorphism $K_*(C^*(\Lambda, \omega_t \circ c)) \cong
K_*(C^*(\Lambda))$ which carries $[s_v]_0$ to $[s_v]_0$ for each $v \in \Lambda^0$ and
carries $[1_{C^*(\Lambda, \omega_t \circ c)}]_0$ to $[1_{C^*(\Lambda)}]_0$ if $\Lambda^0$
is finite.
\end{thm}

\begin{proof}
Let $t \in \RR$. If $\Lambda^0$ is finite then $C^*(\Lambda, \omega_t \circ c)$ has unit $\sum_{v \in \Lambda^0} s_v$.
If $C^*(\Lambda, \omega_t \circ c)$ is unital then arguing as in \cite[Proposition 1.3]{KPR98} it follows
that $\Lambda^0$ is finite.

By rescaling $c$ if necessary, it suffices to prove the second assertion for $t = 1$. We
have $C^*(\Lambda) = C^*(\Lambda, \omega_0 \circ c)$. Corollary~\ref{cor:[0,1]-bundle}
implies that $C^*(\Lambda, \RR, c)|_{[0,1]}$ is a $C([0,1])$-algebra with fibres
$C^*(\Lambda, \RR, c)_t \cong C^*(\Lambda, \omega_t \circ c)$.

Let $\phi : \Lambda \times_d \ZZ^k \to \Lambda$ be the functor $\phi(\lambda, n) =
\lambda$. Then $c \circ \phi \in \Zcat2(\Lambda \times_d \ZZ^k, \RR)$. Lemma~\ref{lem:not
written yet} applied to $\omega_1 \circ c \in \Zcat2(\Lambda, \TT)$ yields an isomorphism
$K_*(C^*(\Lambda, \omega_1 \circ c)) \cong K_*\big(C^*(\Lambda \times_d \ZZ^k, \omega_1
\circ c \circ \phi) \times_{\lt} \ZZ^k\big)$ which carries each $[s_v]_0$ to
$[\jmath_{C^*(\Lambda \times_d \ZZ^k, \omega_t \circ c \circ \phi)}(s_{(v,0)})]_0$.

The map $b(v,m) = m$ from $(\Lambda \times_d \ZZ^k)^0$ to $\ZZ^k$ satisfies $d = \dcat0
b$. Thus Corollary~\ref{cor:[0,1]-bundle} shows that the evaluation maps $C^*(\Lambda,
\RR, c)|_{[0,1]} \to C^*(\Lambda, \RR, c)_t \cong C^*(\Lambda \times_d \ZZ^k, \omega_t
\circ c \circ \phi)$ induce isomorphisms in $K$-theory which carry each
$[\jmath_\Lambda(v)]_0$ to $[s_v]_0$. The universal property of $C^*(\Lambda, \RR, c)$
implies that it carries an action $\lt$ of $\ZZ^k$ characterised by $\lt \circ
\jmath_{\RR} = \jmath_{\RR}$ and $\lt(\jmath_{\Lambda \times_d \ZZ^k}((\lambda, m))) =
\jmath_{\Lambda \times_d \ZZ^k}((\lambda, m + d(\lambda)))$. In particular, $\lt$ fixes
the central copy of $C_0(\RR)$. Lemma~\ref{lem:action field} now implies that the
evaluation maps $C^*(\Lambda, \RR, c) \to C^*(\Lambda, \RR, c)_t$ satisfy the hypotheses
of Theorem~\ref{thm:elliott}. Hence Theorem~\ref{thm:elliott} applied to each of $\pi_1$
and $\pi_0$ gives isomorphisms
\[
K_*(C^*(\Lambda, \omega_1 \circ c))
    \cong K_*(C^*(\Lambda, \RR, c)|_{[0,1]})
    \cong K_*(C^*(\Lambda, \omega_0 \circ c))
    \cong K_*(C^*(\Lambda))
\]
whose composition carries each $[s_v]_0$ to $[s_v]_0$. The final assertion follows from
the first paragraph.
\end{proof}

\begin{example}[$K$-theory of noncommutative tori]
Resume the notation of Example~\ref{eg:nct}, so that $A$ is the free abelian group
generated by  $\{(i,j) : 1 \le j < i \le k\}$ and $c \in \Zcat2(T_k, A)$ is given by
$c(m,n) = \sum_{j < i} m_in_j \cdot (i, j)$. Fix $z : \{(i,j) : j < i \le k\} \to \TT$,
and choose elements $r_{i,j} \in \RR$ such that $e^{\imath r_{i,j}} = z_{i,j}$. There is
a homomorphism $r : A \to \RR$ determined by $(i,j) \mapsto r_{i,j}$, and $b := r \circ
c$ is then an $\RR$-valued $2$-cocycle on $T_k$. For $j < i$, we have $z_{i,j} = e^{\imath
r_{i,j}}$. The resulting character $\chi_z \circ c$ of $A$ was used in
Example~\ref{eg:nct} to prove that the noncommutative torus for $z$ coincides with
$C^*(T_k, \chi_z \circ c)$. We now have $\chi_z \circ c = \omega_1 \circ b$ where
$\omega_1$ is the character $\omega_1(s) = e^{\imath s}$ of $\RR$. Hence
Theorem~\ref{thm:main K-theory} implies that $K_*(C^*(T_k, \chi_z \circ c)) \cong
K_*(C^*(T_k)) = K_*(C(\TT^k))$. Thus we recover Elliott's calculation of the $K$-groups
of the noncommutative tori, which inspired and informed our results here.
\end{example}

\begin{example}[$K$-theory of Heegaard-type quantum spheres]
Let $\Lambda$ be the $2$-graph of Example~\ref{eg:Heegaard bundle}, and resume the
notation used in that example. Let $\Tt$ denote the Toeplitz algebra with generating
isometry $S$. Recall from \cite[Remark~3.3]{HajacMatthesEtAl:ART06} that $C^*(\Lambda)
\cong (\Tt \otimes \Tt)/(\Kk \otimes \Kk)$, giving the following six-term exact sequence.
\[
\begin{tikzpicture}[yscale=0.8]
    \node (00) at (0,0) {$K_1(C^*(\Lambda))$};
    \node (40) at (4,0) {$K_1(\Tt \otimes \Tt)$};
    \node (80) at (8,0) {$K_1(\Kk \otimes \Kk)$};
    \node (82) at (8,2) {$K_0(C^*(\Lambda))$};
    \node (42) at (4,2) {$K_0(\Tt \otimes \Tt)$};
    \node (02) at (0,2) {$K_0(\Kk \otimes \Kk)$};
    \draw[-stealth] (02)--(42);
    \draw[-stealth] (42)--(82);
    \draw[-stealth] (82)--(80);
    \draw[-stealth] (80)--(40);
    \draw[-stealth] (40)--(00);
    \draw[-stealth] (00)--(02);
\end{tikzpicture}
\]
Recall that $(K_0(\Tt), K_1(\Tt), [1]_0) \cong (\ZZ, \{0\}, 1)$ (see
\cite[9.4.2]{Blackadar:K-theory}). The inclusion $\iota : \Kk \to \Tt$ induces the  zero
map on $K_0$ (because $\iota$ carries a minimal projection in $\Kk$ to $S^*S - SS^*$).
Naturality of the K\"unneth formula  (see \cite[23.1.3]{Blackadar:K-theory}) therefore
implies that the map $\ZZ \cong K_0(\Kk \otimes \Kk) \to K_0(\Tt \otimes\Tt) \cong \ZZ$
is the zero map. The K\"unneth formula also implies that $K_1(\Tt \otimes\Tt) = \{0\} =
K_1(\Kk \otimes \Kk)$. So the six-term sequence above becomes
\[
\begin{tikzpicture}
    \node (1) at (-4.5,0) {$0$};
    \node (2) at (-2.5,0) {$K_1(C^*(\Lambda))$};
    \node (3) at (0,0) {$\ZZ$};
    \node (4) at (2,0) {$\ZZ$};
    \node (5) at (4.5,0) {$K_0(C^*(\Lambda))$};
    \node (6) at (6.5,0) {$0$.};
    \draw[-stealth] (1)--(2);
    \draw[-stealth] (2)--(3);
    \draw[-stealth] (3)--(4) node[pos=0.5, above] {\small$0$};
    \draw[-stealth] (4)--(5);
    \draw[-stealth] (5)--(6);
\end{tikzpicture}
\]
Hence $K_j(C^*(\Lambda)) \cong \ZZ$ for $j = 0,1$. Theorem~\ref{thm:main K-theory} then
implies that $K_j(C(S^3_{00\theta})) = \ZZ$ for $j = 0,1$ and $\theta \in [0, 2\pi)$.
Thus our methods recover \cite[Theorem~4.1]{BHMS}.
\end{example}

\subsection*{Kirchberg algebras}\label{sec:Kirchberg}

Recall that a Kirchberg algebra is a purely infinite, simple, nuclear $C^*$-algebra (for
more details see \cite{RordamBook2002}). We invoke the Kirchberg-Phillips theorem, which
classifies Kirchberg algebras, to see that in many cases the isomorphism class of
$C^*(\Lambda, c)$ is independent of $c$. Recall from \cite{EvansSims:JFA12} that a
\emph{generalised cycle with an entrance} in a $k$-graph $\Lambda$ is a pair $(\mu,\nu)
\in \Lambda \times \Lambda$ such that:
\begin{enumerate}
\item $\mu \not= \nu$, $s(\mu) = s(\nu)$, $r(\mu) = r(\nu)$, and $\MCE(\mu\tau, \nu)
    \not= \emptyset$ for every $\tau \in s(\mu)\Lambda$; and
\item there exists $\sigma \in s(\nu)\Lambda$ such that $\MCE(\mu,\nu\sigma) =
    \emptyset$.
\end{enumerate}
The argument of \cite[Lemma~3.7]{EvansSims:JFA12} shows that if $(\mu,\nu)$ is a
generalised cycle with an entrance then $s_{r(\nu)} \ge s_\nu s^*_\nu > s_\mu s^*_\mu
\sim s_\nu s^*_\nu$ and so $s_{r(\nu)}$ is infinite. A vertex $v$ in $\Lambda$ can be
\emph{reached from a generalised cycle with an entrance} if there is a generalised cycle
$(\mu,\nu)$ with an entrance such that $v \Lambda r(\mu) \not= \emptyset$.

For $v,w \in \Lambda^0$ we denote by $v \Lambda w$ the set $\{ \lambda \in \Lambda :
s(\lambda) = v , r(\lambda) = w \}$. Recall from \cite[Proposition~3.6 and
Remark~A.3]{LewinSims:MPCPS10} that a $k$-graph $\Lambda$ is \emph{aperiodic} if,
whenever $\mu \not= \nu$, there exists $\tau \in s(\mu)\Lambda$ such that $\MCE(\mu\tau,
\nu\tau) = \emptyset$, and is \emph{cofinal} if, for all $v,w \in \Lambda^0$, there
exists $m \in \NN^k$ such that $v\Lambda s(\mu) \not= \emptyset$ for all $\mu \in
w\Lambda^m$.

\begin{prop}\label{prp:purely infinite}
Let $\Lambda$ be a row-finite $k$-graph with no sources. Suppose that every vertex in
$\Lambda$ can be reached from a generalised cycle with an entrance and that $\Lambda$ is
aperiodic. For every $c \in \Zcat2(\Lambda, \TT)$, every hereditary subalgebra of
$C^*(\Lambda, c)$ contains an infinite projection. In particular, if $\Lambda$ is
cofinal, then each $C^*(\Lambda, c)$ is simple and purely infinite.
\end{prop}
\begin{proof}
We first show that each $s_v$ is an infinite projection. To see this, fix $v \in
\Lambda^0$. Then there exists $\lambda \in v\Lambda$ and a generalised cycle $(\mu,\nu)$
with an entrance such that $s(\lambda) = r(\mu)$. As discussed above, $s^*_\lambda
s_\lambda = s_{s(\lambda)}$ is infinite, and so $s_v \ge s_\lambda s^*_\lambda \sim
s^*_\lambda s_\lambda$ is also infinite.

The proof of \cite[Proposition 5.3]{BPRS} now shows that every hereditary subalgebra of
$C^*(\Lambda, c)$ contains an infinite projection: the argument generalises to $k$-graph algebras
(see, for example \cite{Sims:CJM06}), and applies without modification to twisted
$k$-graph $C^*$-algebras. If $\Lambda$ is also cofinal, then \cite[Corollary~8.2]{KPS4}
shows that $C^*(\Lambda, c)$ is also simple.
\end{proof}

Recall that for $t \in \RR$, we write $\omega_t$ for the character $s \mapsto e^{\imath
ts}$ of $\RR$.

\begin{cor}
Let $\Lambda$ be an aperiodic, cofinal, row-finite $k$-graph with no sources. Suppose
that $c \in \Zcat2(\Lambda, \RR)$. If every vertex in $\Lambda$ can be reached from a
generalised cycle with an entrance, then $C^*(\Lambda, \omega_t \circ c)$, $t \in \RR$
are mutually isomorphic Kirchberg algebras.
\end{cor}
\begin{proof}
Fix $t \in \RR$. Proposition~\ref{prp:purely infinite} implies that $C^*(\Lambda)$ and
$C^*(\Lambda, \omega_t \circ c)$ are simple and purely infinite. Corollary~8.7 of
\cite{KPS4} implies that they are nuclear and UCT-class. Theorem~\ref{thm:main K-theory}
implies that $C^*(\Lambda)$ and $C^*(\Lambda, \omega_t \circ c)$ are either both unital
or both nonunital, that $K_*(C^*(\Lambda)) \cong K_*(C^*(\Lambda, \omega_t \circ c))$,
and that this isomorphism carries $[1_{C^*(\Lambda)}]_0$ to $[1_{C^*(\Lambda, \omega_t
\circ c)}]_0$ when both $C^*$-algebras are unital. The Kirchberg-Phillips theorem
\cite[Theorem~4.2.4]{Phillips:DM00} (see also \cite[Corollary C]{Kirchberg1994}) then
gives $C^*(\Lambda, \omega_t \circ c) \cong C^*(\Lambda)$.
\end{proof}

\end{document}